\documentclass[reqno]{amsart}
\usepackage{amsfonts,amsmath,amsthm,amssymb,mathrsfs}
\usepackage{color}
\usepackage{amsmath,amssymb,amsfonts,bm}
\usepackage{graphicx}
\usepackage{newunicodechar}
\usepackage{xcolor}
\numberwithin{equation}{section}
\usepackage[pdfstartview=FitH,colorlinks=true]{hyperref}
\hypersetup{urlcolor=red, linkcolor=blue,citecolor=red}
\allowdisplaybreaks

\theoremstyle{plain}

\newtheorem{thm}{Theorem}[section]

\newtheorem{prop}[thm]{Proposition}
\newtheorem{cor}[thm]{Corollary}
\newtheorem{lemma}[thm]{Lemma}
\newtheorem{remark}[thm]{Remark}

\begin{document}
\title[Landau equation]
{Analytic smoothing effect of the time variable\\
for the spatially homogeneous Landau equation
}

\author[C.-J. Xu \& Y. Xu]
{Chao-Jiang Xu and Yan Xu}

\address{Chao-Jiang Xu and Yan Xu
\newline\indent
College of Mathematics and Key Laboratory of Mathematical MIIT,
\newline\indent
Nanjing University of Aeronautics and Astronautics, Nanjing 210016, China
}
\email{xuchaojiang@nuaa.edu.cn; xuyan1@nuaa.edu.cn}

\date{\today}



\begin{abstract}
In this work, we study the Cauchy problem of the spatially homogeneous Landau equation with hard potentials in a close-to-quilibrium framework. 
We prove that the solution to the Cauchy problem enjoys the analytic regularizing effect of the time variable with an $L^2$ initial datum for positive time. So that the smoothing effect of Cauchy problem for the spatially homogeneous Landau equation with hard potentials is exactly same as heat equation.
\end{abstract}

\maketitle

\section{Introduction}
\par In this work, we are concerned the following Cauchy problem of spatially homogenous Landau equation
\begin{equation}\label{1-1}
\left\{
\begin{aligned}
 &\partial_t F=Q(F,F),\\
 &F|_{t=0}=F_0,
\end{aligned}
\right.
\end{equation}
where $F=F(t,v)\ge0$ is the density distribution function at time $t\ge0$, with the velocity variable $v\in\mathbb R^3$. The Landau bilinear collision operator is defined by
$$Q(G,F)(v)=\sum_{i,j=1}^3\partial_i\bigg(\int_{\mathbb R^3}a_{ij}(v-v_*)[G(v_*)\partial_jF(v)-\partial_jG(v_*)F(v)]dv_*\bigg),$$
where
\begin{displaymath}
   a_{ij}(v)=(\delta_{ij}|v|^2-v_iv_j)|v|^\gamma,\quad \gamma\ge-3,
\end{displaymath}
is a symmetric non-negative matrix such that $a_{ij}(v)v_iv_j=0$. Here, $\gamma$ is a parameter which leads to the classification of the hard potential if $\gamma>0$, Maxwellian molecules if $\gamma=0$, soft potential if $\gamma\in]-3, 0[$ and Coulombian potential if $\gamma=-3$.

The Landau equation was introduced as a limit of the Boltzmann equation when the collisions become grazing in~\cite{D-3,V-2}. The global existence, uniqueness of classical solutions for the spatially homogeneous Landau equation with hard potentials, regularizing effects and large-time behavior have been addressed by Desvillettes and Villani~\cite{D-1,V-1}. Moreover, they proved the smoothness of the solution in $C^\infty(]0,\infty[; \mathcal{S}(\mathbb R^3))$. Carrapatoso~\cite{C-5} proved an exponential in time convergence to the equilibrium. In~\cite{C-1}, the authors proved the solution is analytic of $v$ variables for any $t>0$  and the Gevrey regularity in \cite{C-2,C-3}.

Let $\mu$ be the Maxwellian distribution
$$
\mu(v)=(2\pi)^{-\frac{3}{2}}e^{-\frac{|v|^2}{2}},
$$
we shall linearize the Landau equation \eqref{1-1} around $\mu$ with the fluctuation of the density distribution function
\begin{equation*}
F(t,v)=\mu(v)+\sqrt\mu(v)f(t,v),
\end{equation*}
since $Q(\mu,\mu)=0$, the Cauchy problem \eqref{1-1} for $f=f(t,v)$ takes the form
\begin{equation}\label{1-2}
\left\{
\begin{aligned}
 &\partial_t f+\mathcal{L}(f)=\Gamma(f,f),\\
 &f|_{t=0}=f_0,
\end{aligned}
\right.
\end{equation}
with $F_0(v)=\mu+\sqrt\mu f_0(v)$, where
\begin{equation*}
   \Gamma(g,h)=\mu^{\frac{-1}{2}}Q(\mu^{\frac{1}{2}}g,\mu^{\frac{1}{2}}h),
\end{equation*}
\begin{equation*}
    \mathcal{L}(f)=\mathcal{L}_1f+\mathcal{L}_2f,\quad \mathcal{L}_1f=-\Gamma(\mu^{\frac{1}{2}},f),\quad \mathcal{L}_2f=-\Gamma(f,\mu^{\frac{1}{2}}).
\end{equation*}

\par In the Maxwellian molecules case, Villani~\cite{V-1} has proved a linear functional inequality between entropy and entropy dissipation by constructive methods, from which one deduces an exponential convergence of the solution to the Maxwellian equilibrium in relative entropy, which in turn implies an exponential convergence in $L^1$-distance. In~\cite{D-2}, Desvillettes and Villani have proved a functional inequality for entropy dissipation is not linear, from which one obtains a polynomial in time convergence of solutions towards the equilibrium in relative entropy, which implies the same type of convergence in $L^1$-distance.
In~\cite{L-2}, the authors studied the spatially homogeneous Landau equation and non-cutoff Boltzmann equation in a close-to-quilibrium framework and proved the Gelfand-Shilov smoothing effect(see also~\cite{L-3,L-4}). Guo~\cite{G-1} constructed global classical solutions for the spatially inhomogeneous Landau equation near a global Maxwellian in a periodic box, and the smoothness of the solutions have been studied in~\cite{C-4,L-5,H-2}. The analytic smoothing effect of the velocity variable for the nonlinear Landau equation has been studied in~\cite{L-1,M-1}. The variant regularity results in a close to equilibrium setting were considered by~\cite{S-1,C-6,C-7}.

\par Let us give the definition of analytic function spaces $\mathcal{A}(\Omega)$ where $\Omega\subset \mathbb R^n$ is a open domain. We say that $u\in \mathcal{A}(\Omega)$ if $u\in C^\infty(\Omega)$ and there exists a constant $C$ such that for all multi-indices $\alpha\in\mathbb N^n$,
    \begin{equation*}
        \|\partial^\alpha u\|_{L^\infty(\Omega)}\le C^{|\alpha|+1}\alpha!\ .
    \end{equation*}
Remark that, by using the Sobolev embedding, we can replace the $L^\infty$ norm by the $L^2$ norm , or norm in any Sobolev space in the above definition.

\par In this work, we consider the Cauchy problem \eqref{1-2} with $\gamma\ge0$, show that the solution of the Cauchy problem \eqref{1-2} with initial datum in $L^2(\mathbb R^3)$ enjoys the analytic regularizing effect of time variable. Our main result reads as follow.

\begin{thm}\label{theorem}
    Assume $f_0\in L^{2}(\mathbb R^3)$ and $T>0$, let $f$ be the solution of the Cauchy problem \eqref{1-2}  with $\|f\|_{L^\infty([0,T];L^2(\mathbb R^3))}$ small enough. Then there exists a constant $C>0$ such that for any $k\in \mathbb{N}$, we have
    \begin{equation}\label{1-7}
       \|\partial_t^kf(t)\|_{L^2(\mathbb R^3)}\le \frac{C^{k+1}}{t^{k}}k!,\qquad \forall t\in ]0, T].
    \end{equation}
\end{thm}

\begin{remark}
    \rm In the paper~\cite{L-1}, for $f_0\in L^{2}(\mathbb R^3)$ with $\|f\|_{L^\infty([0,T];L^2(\mathbb R^3))}\le\epsilon$
    small enough, the solution of the Cauchy problem \eqref{1-2} satisfies $f(t)\in \mathcal{A}(\mathbb{R}^3)$ for all
    $0< t\le T$, i. e. there exists a constant $C>0$ such that
    \begin{equation*}
         \|t^{\frac{|\alpha|}{2}}\partial_v^\alpha f(t)\|_{L^2(\mathbb R^3)}\le C^{|\alpha|+1}\alpha!,\qquad \forall \alpha\in\mathbb{N}^3, \ \ \forall t\in ]0, T],
    \end{equation*}
which implies that $f\in C^\infty(]0, T[;  \mathcal{A}(\mathbb{R}^3))$, so that we prove only the estimate \eqref{1-7} for the smooth solution of \eqref{1-2}.
Combine with the results of ~\cite{L-1}, we have proved that, if $f$ is the solution of the nonlinear Cauchy problem \eqref{1-2} with $\|f\|_{L^\infty([0,T];L^2(\mathbb R^3))}$ small enough, then we have
$$
f\in \mathcal{A}(]0, T[\times\mathbb{R}^3),
$$
which implies that, the smoothing effect properties of Cauchy problem for the spatially homogeneous Landau equation with hard potentials is exactly same as semilinear heat equation.
\end{remark}

\section{Analysis of Landau collision operator}
The operators $\mathcal{L}_1, \mathcal{L}_2$ and $\Gamma$ are defined in~\cite{G-1} as follow:
\begin{equation}\label{1-3}
    \mathcal{L}_1f=-\sum_{i,j=1}^3\left\{\partial_i[(a_{ij}*\mu)\partial_jf]+(a_{ij}*\mu)\frac{v_i}{2}\frac{v_j}{2}f-\partial_i\left[(a_{ij}*\mu)\frac{v_j}{2}\right]f\right\},
\end{equation}
\begin{equation*}\label{1-4}
    \mathcal{L}_2f=-\sum_{i,j=1}^3\mu^{-\frac{1}{2}}\partial_i\left\{\mu\left[a_{ij}*\left(\mu^{\frac{1}{2}}\partial_jf+\mu^{\frac{1}{2}}\frac{v_j}{2}f\right)\right]\right\},
\end{equation*}
\begin{equation*}\label{1-5}
    \begin{split}
       \Gamma(f,g)&=\sum_{i,j=1}^3\bigg\{\partial_i[(a_{ij}*(\mu^{\frac{1}{2}}f))\partial_jg]-\left[a_{ij}*\left(\frac{v_i}{2}\mu^{\frac{1}{2}}f\right)\right]\partial_jg\\
       &\quad-\partial_i[(a_{ij}*(\mu^{\frac{1}{2}}\partial_jf))g]+\left[a_{ij}*\left(\frac{v_i}{2}\mu^{\frac{1}{2}}\partial_jf\right)\right]g\bigg\}.
    \end{split}
\end{equation*}
Since the use of a different normalization for the Maxwellian, these representations are different in a few places by a factor of $\frac{1}{2}$ from those in~\cite{G-1}.
The linear operator $\mathcal{L}$ is nonnegative.

For later use, we derive some results for the linear operator $\mathcal L$. For simplicity, with $s\in\mathbb R$, we define
\begin{equation*}
    \|f\|_{p,s}=\|(1+|\cdot|)^{s} f\|_{L^p(\mathbb R^3)},1\le p\le\infty,
\end{equation*}
and
\begin{equation*}
   \|f\|^2_{L^2_A}=\sum_{i,j=1}^3\int_{\mathbb R^3}\left(\bar a_{ij}\partial_{i}f\partial_{j}f+\bar a_{ij}\frac{1}{4}v_iv_jf^2\right)dv,
\end{equation*}
where $\bar a_{ij}=a_{ij}*\mu$.
\par From Corollary 1 in~\cite{G-1}, there exists $C_1>0$ such that
\begin{equation*}
    \|f\|_{L^2_A}^2\ge C_1(\|\mathbf P_v\bigtriangledown f\|_{2,\gamma/2}^2+\|(\mathbf I-\mathbf P_v)\bigtriangledown f\|_{2,1+\gamma/2}^2+\|f\|_{2,1+\gamma/2}^2),
\end{equation*}
where for any vector-valued function $g=(g_1,g_2,g_3)$, define the projection to the vector $v\in\mathbb R^3$ as
$$
(\mathbf P_vg)_i=\sum_{j=1}^3g_jv_j\frac{v_i}{|v|^2},\quad1\le i\le3.
$$
Noticing that $f=\mathbf P_v\bigtriangledown f+(\mathbf I-\mathbf P_v)\bigtriangledown f$, we have
\begin{equation}\label{2-1}
    \|f\|_{L^2_A}\ge C_1(\|\bigtriangledown f\|_{2,\gamma/2}+\|f\|_{2,1+\gamma/2}).
\end{equation}
From representation \eqref{1-3}, we can get the coercivity of the operator $\mathcal L_1$.
\begin{lemma}\label{lemma2}
   Let $f\in \mathcal{S}(\mathbb R^3)$, then there exists a constant $C_{2}>0$ such that
   $$
   (\mathcal{L}_1f,f)_{L^2}\ge\|f\|_{L^2_A}^2-C_{2}\|f\|_{2,\gamma/2}^2.
   $$
\end{lemma}
\begin{proof}
    By the representation \eqref{1-3} and integrating by parts, we have
    \begin{displaymath}
        \begin{split}
            (\mathcal{L}_1f,f)_{L^2}&=\sum_{i,j=1}^3\left[((a_{ij}*\mu)\partial_jf,\partial_ig)_{L^2}
            +\frac{1}{4}((a_{ij}*\mu)v_iv_jf,g)_{L^2}\right]\\
            &\qquad\qquad\qquad-\frac{1}{2}\sum_{i,j=1}^3(\partial_i[(a_{ij}*\mu)v_j]f,g)_{L^2}\\
            &=\|f\|_{L^2_A}-\frac{1}{2}\sum_{i,j=1}^3(\partial_i[(a_{ij}*\mu)v_j]f,f)_{L^2}.
        \end{split}
    \end{displaymath}
    Using
    $$\sum_{i=1}^3a_{ij}(v)v_i=\sum_{j=1}^3a_{ij}(v)v_j=0,$$
    it follows that
    \begin{align*}
    \sum_{i,j=1}^3\int_{\mathbb R^3}\partial_i[(a_{ij}*\mu)v_j]f^2dv&=\sum_{i,j=1}^3\int_{\mathbb R^3}\partial_i\bigg(\int_{\mathbb R^3}a_{ij}(v-v')v_j'\mu(v')dv'\bigg)f^2dv\\
    &=\sum_{i,j=1}^3\int_{\mathbb R^3}\partial_i[a_{ij}*(v_j\mu)]f^2dv.
    \end{align*}
    Expanding $\partial_ia_{ij}(v-v')$ to get
    $$\partial_ia_{ij}(v-v')=\partial_ia_{ij}(v)+\sum_{l=1}^3\bigg(\int_0^1\partial_l\partial_ia_{ij}(v-sv')ds\bigg)v_l',$$
    then by
    $$\int_{\mathbb R^3}v'_j\mu(v')dv'=0,$$
    we can deduce that
    $$\partial_ia_{ij}*(v_j\mu)=\sum_{l=1}^3\int_{\mathbb R^3}\int_0^1\partial_l\partial_ia_{ij}(v-sv')dsv_l'v'_j\mu(v')dv',$$
    and using
    $$|\partial^\beta a_{ij}(v)|\le c(1+|v|)^{\gamma+2-|\beta|},\ \forall\beta\in\mathbb N^3,$$
    we can conclude that
    \begin{displaymath}
        \begin{split}
           &\frac{1}{2}\left|\sum_{i,j=1}^3(\partial_i[(a_{ij}*\mu)v_j]f,f)_{L^2}\right|\\
           &=\frac{1}{2}\left|\sum_{i,j=1}^3\int_{\mathbb R^3}\int_{\mathbb R^3}\partial_ia_{ij}(v-v')v'_j\mu(v')dv'f^2(v)dv\right|\\
           &\le\frac{1}{2}\sum_{i,j=1}^3\sum_{l=1}^3\bigg|\int_{\mathbb R^3}\int_{\mathbb R^3}v'_lv'_j\mu(v')\int_0^1\partial_l\partial_ia_{ij}(v-sv')dsdv'f^2(v)dv\bigg|\\
           &\le C_{2}\int_{\mathbb R^3}(1+|v|)^\gamma f^2(v)dv.
        \end{split}
    \end{displaymath}
    We thus complete the proof of the lemma \ref{lemma2}.
\end{proof}
\par We recall the trilinear estimate, which has been addressed in~\cite{L-1}.
\begin{lemma}{\rm(~\cite{L-1}\ )}\label{lemma3}
    Let $F,G,H\in \mathcal{S}(\mathbb R^3)$, then there exists a constant $C_3>0$ such that
    $$|\langle\Gamma(F,G),H\rangle_{L^2}|\le C_3\|F\|_{L^2}\|G\|_{L^2_A}\|H\|_{L^2_A}.$$
\end{lemma}
Let $F=\sqrt\mu,G=f,H=g$ and $F=f,G=\sqrt\mu,H=g$ in lemma \ref{lemma3}, then we have the following estiamtes for the operators $\mathcal{L}_1$ and $\mathcal{L}_2$.
\begin{cor}\label{corollary}
    Let $f,g\in \mathcal{S}(\mathbb R^3)$, then there exists a constant $C_{4}>0$ such that
    $$
    |(\mathcal{L}_1f,g)_{L^2}|\le C_4\|f\|_{L^2_A}\|g\|_{L^2_A},
    $$
    $$
    |(\mathcal{L}_2f,g)_{L^2}|\le C_4\|f\|_{L^2}\|g\|_{L^2_A}.
    $$
\end{cor}

\section{Energy estimates}
For $g\in \mathcal S(\mathbb R^3)$, we need the following interpolation inequality, for all $0<\delta<1$
\begin{equation}\label{interpolation}
      \|g\|^2_{2,\gamma/2}
        \le\delta\|g\|^2_{L^2_A}+C_\delta\|g\|^2_{L^2} .
\end{equation}
From H$\rm\ddot{o}$lder's inequality and inequality \eqref{2-1}, it follows that
\begin{equation*}
\begin{split}
    &\|g\|^2_{2,\gamma/2}=\int_{\mathbb R^3}(1+|v|)^{\gamma}g^{\frac{2\gamma}{\gamma+2}}(v)g^{\frac{4}{\gamma+2}}(v)dv\\
    &\le\|g\|^{\frac{2\gamma}{\gamma+2}}_{2,\gamma/2+1}\|g\|^{\frac{4}{\gamma+2}}_{L^2}\le\left(\frac{1}{C_1}\|g\|_{L^2_A}\right)^{\frac{2\gamma}{\gamma+2}}\|g\|^{\frac{4}{\gamma+2}}_{L^2},\\
\end{split}
\end{equation*}
then by using the Yong inequality  
$$
a b\le\frac{1}{p}a^p+\frac{1}{q}b^{q},\quad (a,b\ge0,\frac{1}{p}+\frac{1}{q}=1)
$$ 
and $\gamma\ge0$, we get
\begin{equation*}
\begin{split}
    \left(\frac{1}{C_1}\|g\|_{L^2_A}\right)^{\frac{2\gamma}{\gamma+2}}\|g\|^{\frac{4}{\gamma+2}}_{L^2}&\le\frac{\gamma}{\gamma+2}\delta\|g\|_{L^2_A}^{2}+\frac{2}{\gamma+2}C_1^{-\gamma}\delta^{-\gamma/2}\|g\|^{2}_{L^2}\\
    &\le\delta\|g\|_{L^2_A}^{2}+C_1^{-\gamma}\delta^{-\gamma/2}\|g\|^{2}_{L^2}.
\end{split}
\end{equation*}
Let $C_\delta=C_1^{-\gamma}\delta^{-\gamma/2}$, then it follows that \eqref{interpolation} holds.

We study now the energy estimates of solution of Cauchy problem \eqref{1-2}, we have

\begin{lemma}\label{lemma-1}
    Assume $f_0\in L^{2}(\mathbb R^3)$ and $T>0$, let $f$ be the solution of the Cauchy problem \eqref{1-2}  with $\|f\|_{L^\infty([0,T];L^2(\mathbb R^3))}$ small enough. Then there exists a constant $B_0>0$ such that
    \begin{equation}\label{3-1+0}
       \|f\|_{L^\infty([0, T];L^2(\mathbb R^3))}^2+\|f\|^2_{L^2([0, T];L^2_A(\mathbb R^3))}\le  B^2_0\|f_0\|^2_{L^2(\mathbb R^3)}\le \epsilon^2 B^2_0.
    \end{equation}
\end{lemma}
We will take $\epsilon$ small such that $0<\epsilon B_0\le 1$.
\begin{proof}
By \eqref{1-2}, we have that
\begin{equation*}
    \frac{1}{2}\frac{d}{dt}\|f\|^2_{L^2}+(\mathcal L_1f,f)_{L^2}=(\Gamma(f,f),f)_{L^2}-(\mathcal L_2f,f)_{L^2}.
\end{equation*}
Using lemma \ref{lemma2} and taking $\delta=\frac{1}{8C_2}$ in \eqref{interpolation}, for all $0\le t\le T$, we can conclude
\begin{equation*}
\begin{split}
    (\mathcal L_1f,f)_{L^2}\ge\|f\|^2_{L^2_A}-C_2\|f\|^2_{2,\gamma/2}\ge\frac{7}{8}\|f\|^2_{L^2_A}-\tilde C_2\|f\|^2_{L^2},
\end{split}
\end{equation*}
and $\tilde C_2$ depends on $C_1$. Since $\|f\|_{L^\infty([0,T];L^2(\mathbb R^3))}\le\epsilon$,
using lemma \ref{lemma3} and taking $\epsilon$ such that $C_3\epsilon\le\frac{1}{8}$, for all $0\le t\le T$, we have
\begin{equation*}
    (\Gamma(f,f),f)_{L^2}\le C_3\|f\|_{L^2}\|f\|^2_{L^2_A}\le\frac{1}{8}\|f\|^2_{L^2_A},
\end{equation*}
corollary \ref{corollary} and H$\rm\ddot{o}$lder's inequality implies
\begin{equation*}
\begin{split}
    \left|(\mathcal L_2f,f)_{L^2}\right|\le C_4\|f\|_{L^2}\|f\|_{L^2_A}\le\frac{1}{8}\|f\|^2_{L^2_A}+2C_4^2\|f\|^2_{L^2}.
\end{split}
\end{equation*}
Combining the above estimates, one has
\begin{equation*}
\begin{split}
     \frac{d}{dt}\|f\|^2_{L^2}+\|f\|^2_{L^2_A}\le\left(2\tilde C_2+4C_4^2\right)\|f\|^2_{L^2},
\end{split}
\end{equation*}
integrating from 0 to $t$ to get
\begin{equation}\label{integration}
    \|f(t)\|^2_{L^2}+\int_0^t\|f(\tau)\|^2_{L^2_A}d\tau\le\left(2\tilde C_2+4C_4^2\right)\int_0^t\|f(\tau)\|^2_{L^2}d\tau,
\end{equation}
then by Gronwall inequality, we get for $0\le t\le T$
\begin{equation}\label{f}
    \|f(t)\|^2_{L^2}\le e^{\left(2\tilde C_2+4C_4^2\right)T}\|f_0\|^2_{L^2}.
\end{equation}
Substituting \eqref{f} into \eqref{integration} and taking $B_0\ge\sqrt{\left(2\tilde C_2+4C_4^2\right)T}e^{2\left(2\tilde C_2+4C_4^2\right)T}$, one can obtain
\begin{equation*}
\begin{split}
    \|f(t)\|^2_{L^2}+\int_0^t\|f(\tau)\|^2_{L^2_A}d\tau\le\left(2\tilde C_2+4C_4^2\right)Te^{\left(2\tilde C_2+4C_4^2\right)T}\|f_0\|^2_{L^2}\le B_0^2\epsilon^2.
\end{split}
\end{equation*}
\end{proof}

\begin{lemma}\label{lemma-1+1}
     Assume $f_0\in L^{2}(\mathbb R^3)$ and $T>0$, let $f$ be the solution of the Cauchy problem \eqref{1-2}  with $\|f\|_{L^\infty([0,T];L^2(\mathbb R^3))}$ small enough. Then there exists a constant $B_1>0$ such that
    \begin{equation}\label{3-1+1}
       \|\tau \partial_\tau f\|_{L^\infty([0, T]; L^2(\mathbb R^3))}^2+\|\tau \partial_\tau f\|^2_{L^2([0, T]; L^2_A(\mathbb R^3))}\le \epsilon^2 B^2_1.
    \end{equation}
\end{lemma}
We also take $\epsilon$ small such that $0<\epsilon B_1\le 1$.
\begin{proof} Since the solution of \eqref{1-2} belongs to $C^\infty(]0, T[; \mathcal{S}(\mathbb{R}^3))$, we have that
$$
\partial_t (t\partial_t f)+\mathcal{L}_1(t\partial_t f)=\partial_t f-\mathcal{L}_2(t\partial_t f)+t\partial_t\Gamma(f, f),
$$
and for $0\le t\le T$
\begin{equation*}
    \begin{split}
&\frac 12\|t\partial_t f\|^2_{L^2}+\int^t_0(\mathcal{L}_1(\tau\partial_\tau f), \tau\partial_\tau f)_{L^2}d\tau \\
=&\int^t_0\tau \|\partial_\tau f\|^2_{L^2} d\tau-\int^t_0
(\mathcal{L}_2(\tau\partial_\tau f), \tau\partial_\tau f)_{L^2} d\tau
+\int^t_0(\tau\partial_\tau\Gamma(f, f),
\tau\partial_\tau f)_{L^2}d\tau\\
=& R_1+ R_2+ R_3.
    \end{split}
\end{equation*}
Firstly, using lemma \ref{lemma2} and \eqref{interpolation} with $\delta=\frac 1{8C_2}$, for all $0\le t\le T$, we can conclude
\begin{equation*}
\begin{split}
        \int^t_0\tau^2(\mathcal{L}_1(\partial_\tau f),\partial_\tau f)_{L^2}d\tau&\ge\|\tau\partial_\tau f\|^2_{L^2([0,t];L^2_A)}-C_2\int_0^t\tau^2\|\partial_\tau f\|^2_{2,\gamma/2}d\tau\\
        &\ge\frac 78\|\tau\partial_\tau f\|^2_{L^2([0,t];L^2_A)}-\tilde C_2 T \int_0^t\tau\|\partial_\tau f\|^2_{L^2}d\tau .
\end{split}
\end{equation*}
For the term $R_1$, since $f$ is solution of \eqref{1-2}, i. e.
$$
\partial_t f=\Gamma(f, f)-\mathcal{L} (f),
$$
using lemma \ref{lemma3} and corollary \ref{corollary}, for all $0\le t\le T$, we have
\begin{equation*}
\begin{split}
         \int^t_0\tau\|\partial_\tau f\|_{L^2}^2d\tau&=\int^t_0\tau(\Gamma(f,f),\partial_\tau f)_{L^2}d\tau-\int^t_0\tau(\mathcal{L}(f),\partial_\tau f)_{L^2}d\tau\\
         &\le C_3\int^t_0\|f\|_{L^2}\|f\|_{L^2_A}\|\tau\partial_\tau f\|_{L^2_A}d\tau\\
         &\qquad+C_4\int^t_0\left(\|f\|_{L^2}+\|f\|_{L^2_A}\right)\|\tau\partial_\tau f\|_{L^2_A}d\tau.
\end{split}
\end{equation*}
Using Cauchy-Schwarz inequality, for $0<\delta<1$,
\begin{equation*}
\begin{split}
         \int^t_0\tau\|\partial_\tau f\|_{L^2}^2d\tau
         &\le \delta \|\tau\partial_\tau f\|^2_{L^2([0,t];L^2_A)}+
         \frac{C^2_3}{2\delta}\|f\|_{L^\infty([0, t]; L^2)}^2\int^t_0\|f\|^2_{L^2_A}d\tau\\
         &\qquad+\frac{C^2_4}{\delta}\left(T\|f\|^2_{L^\infty([0, t]; L^2)}+\int^t_0\|f\|^2_{L^2_A}d\tau\right).
\end{split}
\end{equation*}
Then, \eqref{3-1+0} implies, there exists $C_\delta>0$ such that
\begin{equation}\label{3-2+1}
R_1=\int^t_0\tau\|\partial_\tau f\|_{L^2}^2d\tau \le C_\delta B_0^2\epsilon^2+\delta \|\tau\partial_\tau f\|^2_{L^2([0,t];L^2_A)}.
\end{equation}
For the term $R_2$, using corollary \ref{corollary}, for all $0\le t\le T$, we have
\begin{equation*}
\begin{split}
       |R_2|&= \left|\int^t_0\tau^2(\mathcal{L}_2(\partial_\tau f),\partial_\tau f)_{L^2}d\tau\right|\le C_4\int^t_0\tau^2\|\partial_\tau f\|_{L^2}\|\partial_\tau f\|_{L^2_A}d\tau\\
        &\le\frac{1}{8}\|\tau\partial_\tau f\|^2_{L^2([0,t];L^2_A)}+2C_4^2T\int_0^t\tau\|\partial_\tau f\|^2_{L^2}d\tau,
\end{split}
\end{equation*}
then, using \eqref{3-2+1} with $2C_4^2T \delta\le \frac 18$,
\begin{equation*}
       |R_2|\le\frac{1}{4}\|\tau\partial_\tau f\|^2_{L^2([0,t];L^2_A)}+\tilde C_4B^2_0\epsilon^2.
\end{equation*}
Finally, for the term $R_3$, lemma \ref{lemma3} implies
\begin{equation*}
\begin{split}
        |R_3|=& \left|\int^t_0\tau^2\partial_\tau(\Gamma(f,f),\partial_\tau f)_{L^2}d\tau\right|\\
        &\le\int^t_0\tau^2\left|(\Gamma(\partial_\tau f,f),\partial_\tau f)\right|d\tau
          +\int^t_0\tau^2\left|(\Gamma(f,\partial_\tau f),\partial_\tau f)\right|d\tau\\
         &\le C_3\int^t_0\tau^2\|\partial_\tau f\|_{L^2} \|f\|_{L^2_A} \|\partial_\tau f\|_{L^2_A}d\tau\\
         &\qquad\qquad +C_3\int^t_0\|f\|_{L^2}\|\tau\partial_\tau f\|^2_{L^2_A}d\tau\\
         &\le\frac{1}{8}\|\tau\partial_\tau f\|^2_{L^2([0,t];L^2_A)}+2C^2_3\int^t_0\|f\|^2_{L^2_A}\|\tau\partial_\tau f\|^2_{L^2}d\tau\\
         &\qquad\qquad +C_3\int^t_0\|f\|_{L^2}\|\tau\partial_\tau f\|^2_{L^2_A}d\tau\\
         &\le\frac{1}{8}\|\tau\partial_\tau f\|^2_{L^2([0,t];L^2_A)}+2C^2_3
         \|\tau\partial_\tau f\|^2_{L^\infty([0, t]; L^2)}
         \int^t_0\|f\|^2_{L^2_A}d\tau\\
         &\qquad\qquad +C_3\|f\|^2_{L^\infty([0, t]; L^2)} \int^t_0\|\tau\partial_\tau f\|^2_{L^2_A}d\tau.
    \end{split}
\end{equation*}
Using \eqref{3-1+0} and taking $\epsilon>0$ small such that
$$
2C^2_3 B^2_0\epsilon^2\le \frac 14,\quad C_3 B^2_0\epsilon^2\le \frac 18.
$$
We get then, for all $0\le t\le T$,
\begin{equation*}
         \left|\int^t_0\tau^2\partial_\tau(\Gamma(f,f),\partial_\tau f)_{L^2}d\tau\right|
         \le\frac{1}{4}\|\tau\partial_\tau f\|^2_{L^2([0,t];L^2_A)}+\frac 14 \|\tau\partial_\tau f\|^2_{L^\infty([0, t]; L^2)} .
\end{equation*}
Combining the above estimates, taking $\delta=\frac 18$ in \eqref{3-2+1}, one has
\begin{equation*}
\begin{split}
        \frac 14 \|\tau\partial_\tau f\|^2_{L^\infty([0, T]; L^2)} +\frac{3}{8}\|\tau\partial_\tau f\|^2_{L^2([0,T];L^2_A)}\le C_5\epsilon^2+\tilde C_2 T \int_0^T\tau\|\partial_\tau f\|^2_{L^2}d\tau,
\end{split}
\end{equation*}
using \eqref{3-2+1} with $\tilde C_2 T \delta\le \frac 18$ and taking $B_1\ge2\sqrt{C_5}$, then it follows that
    \begin{equation*}
     \|\tau\partial_\tau f\|^2_{L^\infty([0, T]; L^2)} +\|\tau\partial_\tau f\|^2_{L^2([0,T];L^2_A)}\le4C_5\epsilon^2\le B^2_1\epsilon^2,
    \end{equation*}
with $B_1$ depends only on $C_1, C_2, C_3, C_4$ and $T$, which end the proof of lemma \ref{lemma-1+1}.
\end{proof}

\section{Proof of main Theorem}
In this section, we shall show the analytic regularity of time variable for $t>0$. We construct the following estimate, from which we can deduce the inequality \eqref{1-7} directly.

\begin{prop}\label{Proposition4-1}
     Assume $f_0\in L^{2}(\mathbb R^3)$ and $T>0$, let $f$ be the solution of the Cauchy problem \eqref{1-2}  with $\|f\|_{L^\infty([0,T];L^2(\mathbb R^3))}$ small enough. Then there exists a constant $B>0$ such that, for any $k\in\mathbb{N}_+$
    \begin{equation}\label{3-1}
       \|\tau^k\partial^k_\tau f\|_{L^\infty([0,T];L^2(\mathbb R^3))}^2+
       \|\tau^k\partial^k_\tau f\|^2_{L^2([0,T];L^2_A(\mathbb R^3))}\le B^{2(k-1)}((k-2)!)^2.
    \end{equation}
\end{prop}

We have that \eqref{3-1} implies immediately \eqref{1-7}, so it is enough to prove this proposition \ref{Proposition4-1} for theorem \ref{theorem}. We prove this proposition by induction for the index $k$. For $k=1$, it is enough to take, in \eqref{3-1+1},
$$
0<\epsilon B_1\le 1,
$$
and by convention $(-1)!=1, 0!=1$.
Now for $k\ge 2$, Since $\mu$ is a function with respect to the variable $v$, we have
$$
t^k\partial^k_t\mathcal{L}f=\mathcal{L}(t^k\partial^k_tf).
$$
Then by \eqref{1-2}, one can obtain,
\begin{equation*}
\begin{split}
        \partial_t (t^k\partial^k_t f)&+\mathcal{L}_1(t^k\partial^k_t f)
        =k t^{k-1}\partial^k_t f-\mathcal{L}_2(t^k\partial^k_t f)+\Gamma(f, t^k\partial^k_t f)\\
        &+\Gamma(t^k\partial^k_t f, f)+\sum_{1\le j\le k-1}C^j_k\ \Gamma(t^{j}\partial^{j}_tf, t^{k-j}\partial^{k-j}_t f),
\end{split}
\end{equation*}
where $C^j_k=\frac{k!}{j!(k-j)!}$. Then taking the $L^2(\mathbb{R}^3)$ inner product of both sides with respect to $t^k\partial^k_tf$, we have
\begin{equation*}
\begin{split}
&\frac{1}{2}\frac{d}{dt}\|t^k\partial^k_tf\|_{L^2}^2+(\mathcal{L}_1(t^{k}\partial^k_tf), t^{k}\partial^k_tf)_{L^2}\\
&\qquad \qquad =kt^{2k-1}\|\partial^k_tf\|_{L^2}^2-(\mathcal{L}_2(t^{k}\partial^k_tf), t^{k}\partial^k_tf)_{L^2}\\
&\qquad \qquad +(\Gamma(f, t^k\partial^k_tf), t^k\partial^k_tf)_{L^2} +(\Gamma(t^k\partial^k_tf, f), t^k\partial^k_tf)_{L^2}
\\
&\qquad \qquad +\sum_{1\le j\le k-1}C^j_k\ \Gamma(t^{j}\partial^{j}_tf, t^{k-j}\partial^{k-j}_t f), t^k\partial^k_tf)_{L^2}.
\end{split}
\end{equation*}
For all $0< t\le T$, integrating from 0 to $t$, using lemma \ref{lemma2} and \eqref{interpolation} with $\delta=\frac 1{8C_2}$, we can conclude
\begin{equation*}
\begin{split}
        \int^t_0\tau^{2k}(\mathcal{L}_1(\partial^k_\tau f),\partial^k_\tau f)_{L^2}d\tau&\ge\|\tau^k\partial^k_\tau f\|^2_{L^2([0,t];L^2_A)}-C_2\int_0^t\tau^{2k}\|\partial^k_\tau f\|^2_{2,\gamma/2}d\tau\\
        &\ge\frac{7}{8}\|\tau^k\partial^k_\tau f\|^2_{L^2([0,t];L^2_A)}-\tilde C_2 \int_0^t\tau^{2k}\|\partial^k_\tau f\|^2_{L^2}d\tau.
\end{split}
\end{equation*}
Using corollary \ref{corollary}, for all $0\le t\le T$, we have
\begin{equation*}
\begin{split}
       &\left|\int^t_0\tau^{2k}(\mathcal{L}_2(\partial^k_\tau f),\partial^k_\tau f)_{L^2}d\tau\right|\le C_4\int^t_0\tau^{2k}\|\partial^k_\tau f\|_{L^2}\|\partial^k_\tau f\|_{L^2_A}d\tau\\
        &\qquad \le\frac{1}{8}\|\tau^k\partial^k_\tau f\|^2_{L^2([0,t];L^2_A)}+2C_4^2\int_0^t\tau^{2k}\|\partial^k_\tau f\|^2_{L^2}d\tau.
\end{split}
\end{equation*}
Finally, using  lemma \ref{lemma3}, we have
\begin{equation}\label{4-1+1}
\begin{split}
&\frac{1}{2}\|t^k\partial^k_t f\|_{L^2}^2+\frac{3}{4}\int^t_0\|\tau^k\partial^k_\tau f\|^2_{L^2_A}d\tau\\
&\qquad \le
k\int^t_0\tau^{2k-1}\|\partial^k_\tau f\|_{L^2}^2d\tau+\tilde C_3 \int_0^t\|\tau^{k}\partial^k_\tau f\|^2_{L^2}d\tau\\
        &\qquad\qquad +C_3\sum_{0\le j\le k}C^j_k\ \int^t_0\|\tau^{j}\partial^{j}_\tau f\|_{L^2}\|\tau^{k-j}\partial^{k-j}_\tau f\|_{L^2_A}\|\tau^k\partial^k_\tau f\|_{L^2_A}d\tau,
\end{split}
\end{equation}
with $\tilde C_3=\tilde C_2+2C_4^2$ depends only on $C_1, C_2, C_3, C_4$ and $T$.

We prove now \eqref{3-1} by induction on $k$.  Assume that for $k\ge 2$, \eqref{3-1} holds true for $1\le m\le k-1$,
    \begin{equation}\label{4-3+0}
       \|\tau^m\partial^m_\tau f\|_{L^\infty([0, T];L^2(\mathbb R^3))}^2+\|\tau^m\partial^m_\tau f\|^2_{L^2([0, T]; L^2_A(\mathbb R^3))}\le B^{2(m-1)}((m-2)!)^2.
    \end{equation}
And we shall prove that \eqref{3-1} holds true for $m=k$. We estimate the terms of right hand side of \eqref{4-1+1} by the following lemmas.

\begin{lemma}\label{lemma4-1}
Assume that \eqref{4-3+0} holds true for any $1\le m\le k-1$, and $f$ satisfies \eqref{3-1+0}, then
\begin{equation}\label{4-3+1}
k\int^t_0\tau^{2k-1}\|\partial^k_\tau f\|_{L^2}^2d\tau \le\frac{1}{8}\|\tau^k\partial^k_\tau f\|^2_{L^2([0,t];L^2_A)}+A_1 B^{2(k-2)}((k-2)!)^2,
\end{equation}
with $A_1$ depends only on $C_1, C_2, C_3, C_4$ and $T$.
\end{lemma}
We have also
\begin{lemma}\label{lemma4-2}
Assume that \eqref{4-3+0} holds true for any $1\le m\le k-1$, then
\begin{equation}\label{4-3+2}
\begin{split}
C_3\sum_{1\le j\le k-1}C^j_k\ \int^t_0\|\tau^{j}\partial^{j}_\tau f\|_{L^2}\|\tau^{k-j}\partial^{k-j}_\tau f\|_{L^2_A}\|\tau^k\partial^k_\tau f\|_{L^2_A}d\tau\\
\le \frac{1}{8}\|\tau^k\partial^k_\tau f\|^2_{L^2([0,t];L^2_A)}+A_2 B^{2(k-2)}((k-2)!)^2,
\end{split}
\end{equation}
with $A_2$ depends only on $C_1, C_2, C_3, C_4$ and $T$.
\end{lemma}
And
\begin{lemma}\label{lemma4-3}
Assume that $f$ satisfies \eqref{3-1+0}, then, for $0<t\le T$,
\begin{equation}\label{4-3+3}
\begin{split}
&C_3\ \int^t_0\|\tau^{k}\partial^{k}_\tau f\|_{L^2}\| f\|_{L^2_A}\|\tau^k\partial^k_\tau f\|_{L^2_A}d\tau\\
&\le \frac{1}{8}\|\tau^k\partial^k_\tau f\|^2_{L^2([0,t];L^2_A)}+2C^2_3 B^2_0\epsilon^2\
\|\tau^{k}\partial^{k}_\tau f\|^2_{L^\infty([0, t]; L^2)},
\end{split}
\end{equation}
and
\begin{equation}\label{4-3+4}
\ \int^t_0\| f\|_{L^2}\|\tau^{k}\partial^{k}_\tau f\|^2_{L^2_A}d\tau\\
\le B_0\epsilon\|\tau^k\partial^k_\tau f\|^2_{L^2([0,t];L^2_A)}.
\end{equation}
\end{lemma}
We will give the proofs of these three lemmas in the next section.

\bigskip
\noindent
{\bf End of proof of Proposition \ref{Proposition4-1} }

Choose $0<\epsilon<1$ small such that
$$
C_3 B_0\epsilon\le \frac{1}{8},\quad 2C^2_3 B^2_0\epsilon^2\le \frac 14.
$$
Assume that \eqref{4-3+0} holds true for any $1\le m\le k-1$, and $f$ satisfies \eqref{3-1+0}, then combine \eqref{4-1+1}, \eqref{4-3+1}, \eqref{4-3+2},\eqref{4-3+3},\eqref{4-3+4}, we get, for $0<t\le T$,
\begin{equation*}
\begin{split}
&\|t^k\partial^k_t f\|_{L^2}^2+\int^t_0\|\tau^k\partial^k_\tau f\|^2_{L^2_A}d\tau\\
&\qquad \le
4(A_1+A_2) (B^{k-2}(k-2)!)^2+4\tilde C_3 \int_0^t\|\tau^{k}\partial^k_\tau f\|^2_{L^2}d\tau,
\end{split}
\end{equation*}
with $\tilde C_3$ depends only on $C_1, C_2, C_3, C_4$ and $T$.
By using Gronwall inequality, we get for $0<t\le T$,
\begin{equation*}
\|t^k\partial^k_t f\|_{L^2}^2 \le 4 e^{4\tilde C_3 T}(A_1+A_2) B^{2(k-2)}((k-2)!)^2,
\end{equation*}
which deduce
\begin{align*}
&\|\tau^k\partial^k_\tau f\|_{L^\infty ([0, T], L^2)}^2+\|\tau^k\partial^k_\tau f\|^2_{L^2([0, T], L^2_A)}\\
&\qquad \le
4( e^{4\tilde C_3 T}\ 4\tilde C_3 T+1)(A_1+A_2) B^{2(k-2)}((k-2)!)^2.
\end{align*}
We prove then
\begin{equation*}
\|\tau^k\partial^k_\tau f\|_{L^\infty ([0, T], L^2)}^2+\|\tau^k\partial^k_\tau f\|^2_{L^2([0, T], L^2_A)}
\le B^{2(k-1)}((k-2)!)^2,
\end{equation*}
if we choose the constant $B$ such that
$$
4(e^{4\tilde C_3 T} \ 4\tilde C_3 T+1)(A_1+A_2) \le B^2,
$$
so that the constant $B$ depends only on $C_1, C_2, C_3, C_4, T$ and small $\epsilon$. We finish the proof of proposition \ref{Proposition4-1}.

\section{Proofs of technical Lemmas}
Before give the proof of lemma \ref{lemma4-1}, lemma \ref{lemma4-2} and lemma \ref{lemma4-3},
we need the following lemma
\begin{lemma}
    For all $k\in\mathbb N, k\ge5$, we have
\begin{equation}\label{summation}
\begin{split}
    \sum_{2\le j\le k-3}\frac{k(k-1)}{j(j-1)(k-j-1)(k-j-2)}\le 12.
\end{split}
\end{equation}
\end{lemma}
\begin{proof}
   Without loss of generality, we may assume $k-1$ is even, then the summation can be rewritten as
   \begin{equation*}
   \begin{split}
      \sum_{2\le j\le \frac{k-3}{2}}\frac{k(k-1)}{j(j-1)(k-1-j)(k-2-j)}+\sum_{\frac{k-1}{2}\le j\le k-3}\frac{k(k-1)}{j(j-1)(k-1-j)(k-2-j)}.
    \end{split}
    \end{equation*}
    For the first term in above, since $j\le\frac{k-3}{2}$, we have $k-j\ge\frac{k+3}{2}$. Then it follows that
    \begin{equation*}
      \sum_{2\le j\le \frac{k-3}{2}}\frac{k(k-1)}{j(j-1)(k-1-j)(k-2-j)}\le\sum_{2\le j\le \frac{k-3}{2}}\frac{4}{j(j-1)}\le4.
    \end{equation*}
    For the second term, by $j\ge\frac{k-1}{2}$, we have
    \begin{equation*}
      \sum_{\frac{k-1}{2}\le j\le k-3}\frac{k(k-1)}{j(j-1)(k-1-j)(k-2-j)}\le\sum_{\frac{k-1}{2}\le j\le k-3}\frac{8}{(k-1-j)(k-2-j)}\le8.
    \end{equation*}
    Thus \eqref{summation} holds true.
\end{proof}

\noindent
{\bf Proof of Lemma \ref{lemma4-1} }

For $k\ge2$, by \eqref{1-2}, one has
\begin{align*}
\partial_t^k f&=\partial_t (\partial_t^{k-1}f)=-\mathcal L(\partial_t^{k-1} f)
+\partial_t^{k-1} \Gamma(f, f)\\
&=-\mathcal L(\partial_t^{k-1} f)+\sum_{0\le j\le k-1}C^j_{k-1}\Gamma(\partial^j_tf,\partial^{k-1-j}_tf).
\end{align*}
Then we have
\begin{equation*}
\begin{split}
     k\int_0^t\tau^{2k-1}\|\partial_\tau^k f\|^2_{L^2}d\tau=&k\sum_{0\le j\le k-1}C^j_{k-1}\int_0^t\tau^{2k-1}(\Gamma(\partial^j_\tau f,\partial^{k-1-j}_\tau f),\partial_\tau^{k} f)_{L^2}\\
     &-k\int_0^t\tau^{2k-1}(\mathcal L(\partial_\tau^{k-1} f),\partial_\tau^{k} f)_{L^2}d\tau\\
     =&S_1+S_2.
\end{split}
\end{equation*}
Using lemma \ref{lemma3}, we can conclude
\begin{equation*}
\begin{split}
    |S_1|&\le C_3k\sum_{0\le j\le k-1}C^j_{k-1}\int_0^t\|\tau^j\partial^j_\tau f\|_{L^2}\|\tau^{k-1-j}\partial^{k-1-j}_\tau f\|_{L^2_A}\|\tau^k\partial^k_\tau f\|_{L^2_A}d\tau\\
    &\le4C_3^2k^2\left(\sum_{0\le j\le k-1}C^j_{k-1}\|\tau^j\partial^j_\tau f\|_{L^\infty([0,t];L^2)}\|\tau^{k-1-j}\partial^{k-1-j}_\tau f\|_{L^2([0,t];L^2_A)}\right)^2\\
    &\qquad+\frac{1}{16}\|\tau^k\partial^k_\tau f\|^2_{L^2([0,t];L^2_A)}.
\end{split}
\end{equation*}
From \eqref{4-3+0}, one can obtain
\begin{equation}\label{5-1}
\begin{split}
    &\sum_{0\le j\le k-1}C^j_{k-1}\|\tau^j\partial^j_\tau f\|_{L^\infty([0,t];L^2)}\|\tau^{k-1-j}\partial^{k-1-j}_\tau f\|_{L^2([0,t];L^2_A)}\\
    &\le\sum_{0\le j\le k-1}C^j_{k-1}B^{j-1}(j-2)!B^{k-2-j}(k-3-j)!\\
    &\le B^{k-3}(k-3!)\left(\sum_{2\le j\le k-3}\frac{k(k-1)}{j(j-1)(k-1-j)(k-2-j)}+6\right).
\end{split}
\end{equation}
Substituting \eqref{summation} into \eqref{5-1} we get
\begin{equation*}
\begin{split}
    &\sum_{0\le j\le k-1}C^j_{k-1}\|\tau^j\partial^j_\tau f\|_{L^\infty([0,t];L^2)}\|\tau^{k-1-j}\partial^{k-1-j}_\tau f\|_{L^2([0,t];L^2_A)}\le18B^{k-3}(k-3)!,
\end{split}
\end{equation*}
from which we can conclude
\begin{equation*}
    |S_1|\le\frac{1}{16}\|\tau^k\partial^k_\tau f\|^2_{L^2([0,t];L^2_A)}+C_6\left(B^{k-3}(k-2)!\right)^2,
\end{equation*}
with $C_6$ depends only on $C_1, C_2, C_3, C_4$ and $T$, where we use $\frac{k}{k-2}\le 3$.

For the term $S_2$ ,using corollary \ref{corollary} and \eqref{4-3+0}, we have
\begin{equation*}
\begin{split}
   |S_2|&\le C_4k\int_0^t\left(\|\tau^{k-1}\partial^{k-1}_\tau f\|_{L^2}+\|\tau^{k-1}\partial^{k-1}_\tau f\|_{L^2_A}\right)\|\tau^{k}\partial^{k}_\tau f\|_{L^2_A}d\tau\\
   &\le4C_4^2k^2\left(T\|\tau^{k-1}\partial^{k-1}_\tau f\|^2_{L^\infty([0,t];L^2)}+\|\tau^{k-1}\partial^{k-1}_\tau f\|^2_{L^2([0,t];L^2_A)}\right)\\
   &\qquad+\frac{1}{16}\|\tau^k\partial^k_\tau f\|^2_{L^2([0,t];L^2_A)}\\
   &\le4C_4^2k^2(T+1)\left(B^{k-2}(k-3)!\right)^2+\frac{1}{16}\|\tau^k\partial^k_\tau f\|^2_{L^2([0,t];L^2_A)}\\
   &\le C_7\left(B^{k-2}(k-2)!\right)^2+\frac{1}{16}\|\tau^k\partial^k_\tau f\|^2_{L^2([0,t];L^2_A)},
\end{split}
\end{equation*}
with $C_7$ depends only on $C_1, C_2, C_3, C_4$ and $T$.

Taking $A_1=C_6+C_7$, so that $A_1$ depends only on $C_1, C_2, C_3, C_4$ and $T$, then combining $S_1$ and $S_2$, we get
\begin{equation*}
    k\int_0^t\tau^{2k-1}\|\partial^k_\tau f\|^2_{L^2}d\tau\le\frac{1}{8}\|\tau^k\partial^k_\tau f\|^2_{L^2([0,t];L^2_A)}+A_1\left(B^{k-2}(k-2)!\right)^2.
\end{equation*}

\vspace{1cm}
\noindent
{\bf Proof of Lemma \ref{lemma4-2} }

Using H$\ddot{\rm o}$lder's inequality and \eqref{4-3+0}, we have
\begin{equation*}
\begin{split}
    &C_3\sum_{1\le j\le k-1}C^j_k\int_0^t\|\tau^j\partial^j_\tau f\|_{L^2}\|\tau^{k-j}\partial^{k-j}_\tau f\|_{L^2_A}\|\tau^{k}\partial^{k}_\tau f\|_{L^2_A}d\tau\\
    &\le2C_3^2\left(\sum_{1\le j\le k-1}C^j_k\|\tau^j\partial^j_\tau f\|_{L^\infty([0,t];L^2)}\|\tau^{k-j}\partial^{k-j}_\tau f\|_{L^2([0,t];L^2_A)}\right)^2\\
    &\qquad+\frac{1}{8}\|\tau^{k}\partial^{k}_\tau f\|^2_{L^2([0,t];L^2_A)}\\
    &\le2C_3^2\left(B^{k-2}(k-2)!\right)^2\left(\sum_{2\le j\le k-3}\frac{k(k-1)}{j(j-1)(k-j)(k-j-1)}+6\right)^2\\
    &\qquad+\frac{1}{8}\|\tau^{k}\partial^{k}_\tau f\|^2_{L^2([0,t];L^2_A)}\\
    &\le2C_3^2\left(B^{k-2}(k-2)!\right)^2\left(\sum_{2\le j\le k-3}\frac{k(k-1)}{j(j-1)(k-j-1)(k-j-2)}+6\right)^2\\
    &\qquad+\frac{1}{8}\|\tau^{k}\partial^{k}_\tau f\|^2_{L^2([0,t];L^2_A)}.
\end{split}
\end{equation*}
Then from \eqref{summation}, we can get
\begin{equation*}
\begin{split}
   &C_3\sum_{1\le j\le k-1}C^j_k\int_0^t\|\tau^j\partial^j_\tau f\|_{L^2}\|\tau^{k-j}\partial^{k-j}_\tau f\|_{L^2_A}\|\tau^{k}\partial^{k}_\tau f\|_{L^2_A}d\tau\\
   &\le\frac{1}{8}\|\tau^{k}\partial^{k}_\tau f\|^2_{L^2([0,t];L^2_A)}+A_2\left(B^{k-2}(k-2)!\right)^2,
\end{split}
\end{equation*}
with $A_2$ depends only on $C_1, C_2, C_3, C_4$ and $T$.

\vspace{1cm}

\noindent
{\bf Proof of Lemma \ref{lemma4-3} }

For the inequality \eqref{4-3+3}, using H$\ddot{\rm o}$lder's inequality and \eqref{3-1+0}, one has
\begin{equation*}
\begin{split}
    &C_3\int_0^t\|\tau^k\partial^k_\tau f\|_{L^2}\|f\|_{L^2_A}\|\tau^k\partial^k_\tau f\|_{L^2_A}d\tau\\
    &\le\frac{1}{8}\|\tau^k\partial^k_\tau f\|^2_{L^2([0,t];L^2_A)}+2C_3^2\|\tau^k\partial^k_\tau f\|^2_{L^\infty([0,t];L^2)}\int_0^t\|f\|^2_{L^2_A}d\tau\\
    &\le\frac{1}{8}\|\tau^k\partial^k_\tau f\|^2_{L^2([0,t];L^2_A)}+2C_3^2B_0^2\epsilon^2.
\end{split}
\end{equation*}
For the inequality \eqref{4-3+4}, the inequality \eqref{3-1+0} implies
\begin{equation*}
\begin{split}
    \int_0^t\|f\|_{L^2}\|\tau^k\partial^k_\tau f\|^2_{L^2_A}d\tau&\le \|f\|_{L^\infty([0,T];L^2)}\|\tau^k\partial^k_\tau f\|^2_{L^2([0,t];L^2_A)}\\
    &\le B_0\epsilon\|\tau^k\partial^k_\tau f\|^2_{L^2([0,t];L^2_A)}.
\end{split}
\end{equation*}

\bigskip
\noindent {\bf Acknowledgements.}
This work was supported by the NSFC (No.12031006) and the Fundamental
Research Funds for the Central Universities of China.


\begin{thebibliography}{99}

\bibitem{C-1} Chen H., Li W.-X. and Xu C.-J., Analytic smoothness effect of solutions for spatially homogeneous Landau equation. {\em J. Differential Equations}, 248 (2010), 77-94. 

\bibitem{C-2} Chen H., Li W.-X. and Xu C.-J., Propagation of Gevrey regularity for solutions of Landau equations. {\em Kinet. Relat. Models}, 1 (2008), 355-368.

\bibitem{C-3} Chen H., Li W.-X. and Xu C.-J., Gevrey regularity for solutions of the spatially homogeneous Landau equation. {\em Acta Math. Sci. Ser. B (Engl. Ed.)}, 29 (2009), 673-686. 

\bibitem{C-4} Chen Y. M, Desvillettes L. and He L. B., Smoothing Effects for Classical Solutions of the Full Landau Equation. {\em Arch. Ration. Mech. Anal.}, 193 (2009), 21-55. 

\bibitem{C-5} Carrapatoso K., Exponential convergence to equilibrium for the homogeneous Landau equation with hard potentials. {\em Bull. Sci. Math.}, 139 (2015), 777-805. 

\bibitem{C-6} Carrapatoso K., Tristani I. and  Wu K. C., Cauchy problem and exponential stability for the inhomogeneous Landau equation. {\em Arch. Ration. Mech. Anal.}, 221 (2016), 363-418. 

\bibitem{C-7} Carrapatoso K. and  Mischler S., Landau equation for very soft and Coulomb potentials near Maxwellians. {\em Ann. PDE },
3 (2017), Paper No. 1, 65 pp. 

\bibitem{D-1} Desvillettes L. and Villani C., On the spatially homogeneous landau equation for hard potentials part i: existence, uniqueness and smoothness. {\em Communications in Partial Differential Equations}, 25 (2000), 179-259.

\bibitem{D-2} Desvillettes L. and Villani C., On the spatially homogeneous landau equation for hard potentials part ii: h-theorem and applications. {\em Communications in Partial Differential Equations}, 25 (2000), 261-298.

\bibitem{D-3} Desvillettes L., On asymptotics of the Boltzmann equation when the collisions become grazing. {\em Transport Theory Statist. Phys.}, 21 (1992), 259-276. 

\bibitem{G-1} Guo Y., The Landau Equation in a Periodic Box. {\em Comm. Math. Phys.}, 231 (2002), 391-434.

\bibitem{H-2} Henderson C. and Snelson S., $C^\infty$ Smoothing for Weak Solutions of the Inhomogeneous Landau Equation. {\em Arch. Ration. Mech. Anal.}, 236 (2020), 113-143. 

\bibitem{L-2} Lerner N.,  Morimoto Y.,  Pravda-Starov K. and Xu C.-J., Phase space analysis and functional calculus for the linearized Landau and Boltzmann operators. {\em Kinet. Relat. Models}, 6 (2013), 625-648. 

\bibitem{L-1} Li H. G. and Xu C.-J., The analytic smoothing effect of solutions for the nonlinear spatially homogeneous Landau equation with hard potentials.{\em  Sci China Math.}, 2021, 64. https://doi.org/10.1007/s11425-021-1888-6


\bibitem{L-3} Li H. G. and Xu C.-J., Cauchy problem for the spatially homogeneous Landau equation with Shubin class initial datum and Gelfand-Shilov smoothing effect. {\em SIAM J. Math. Anal.}, 51 (2019), 532-564. 

\bibitem{L-5} Liu S. and  Ma X., Regularizing effects for the classical solutions to the Landau equation in the whole space. {\em J. Math. Anal. Appl.}, 417 (2014), 123-143. 

\bibitem{L-4} Morimoto Y., Pravda-Starov K. and Xu C.-J., A remark on the ultra-analysis smoothing properties of the spatially homogeneous Landau equation. {\em  Kinet. Relat. Models}, 6 (2013), 715-727.

\bibitem{M-1} Morimoto Y. and Xu C.-J., Analytic smoothing effect for the nonlinear Landau equation of Maxwellian molecules. {\em Kinet. Relat. Models}, 13 (2020), 951-978. 

\bibitem{S-1} Strain R. M. and  Guo Y., Exponential Decay for Soft Potentials near Maxwellian. {\em Arch. Ration. Mech. Anal.}, 187 (2008), 287-339. 

\bibitem{V-1} Villani C., On the spatially homogeneous Landau equation for Maxwellian molecules. {\em Math. Models Methods Appl. Sci.}, 8 (1998), 957-983. 

\bibitem{V-2} Villani C., On a New Class of Weak Solutions to the Spatially Homogeneous Boltzmann and Landau Equations.{\em Arch. Rational Mech. Anal.}, 143 (1998), 273-307. 
\end{thebibliography}
\end{document}